\documentclass[12pt,fleqn]{article}
\usepackage{amsmath, amsthm,enumerate,fleqn,amssymb,}
\usepackage{titlesec,soul}
\usepackage[margin=1in]{geometry}
\titleformat{\section}[block]{\centering\bfseries}{\thesection.}{1em}{}
\titlespacing*{\section}{36pt}{6pt}{6pt}
\titleformat{\subsection}[block]{\centering\bfseries}{\thesection.}{1em}{}
\titlespacing*{\subsection}{24pt}{6pt}{6pt}
\newtheoremstyle{theorem}{10pt}{10pt}{\sl}{\parindent}{\bf}{. }{ }{}
\theoremstyle{theorem}
\newtheorem{theorem}{Theorem}[section]

\newtheorem{lemma}[theorem]{Lemma}
\newtheorem*{Zlc}{Zalcman's Lemma}
\newtheoremstyle{defi}{10pt}{10pt}{\rm}{\parindent}{\bf}{. }{ }{}
\theoremstyle{defi}
\newtheorem{definition}[theorem]{Definition}
\newtheorem{example}[theorem]{Example}

\numberwithin{equation}{section}
\date{}
\newcommand{\fr}{\mathcal{F}}

\newcommand{\C}{\mathbb{C}}

\newcommand{\N}{\mathbb{N}}

\begin{document}
\title{\normalsize\bf Quasi-normality and Shared Functions}
\author{\footnotesize  Gopal DATT  \\
\footnotesize  Department of Mathematics,\\[-2pt]
\footnotesize  University of Delhi, Delhi 110007, India\\[-2pt]
\footnotesize  e-mail: ggopal.datt@gmail.com}
\maketitle
\begin{center}
\begin{minipage}{\textwidth}\fontsize{10}{10}\selectfont
\textbf{Abstract:} We prove two sufficient conditions of quasi-normality in which  each pair $f$ and $g$ of  $\fr$  shares some holomorphic functions.
\end{minipage}
\end{center}

\jot6pt
\vskip3em
\thispagestyle{empty}
\baselineskip14pt\fontsize{11}{14}\selectfont
\section{Introduction}
Inspired by the heuristic principle, attributed to Bloch, Schwick \cite{Sch 92} discovered a connection between normality and shared values. Since then many researchers have given various sufficient conditions of normality using shared values. But a less number of articles are there in support of quasi-normal family. In our best knowledge there are a few results connecting quasi-normality and shared values. In this article we prove a quasi-normality criterion concerning shared functions. In 1975, Zalcman \cite{Zalc 1} proved a remarkable  result, now known as Zalcman's Lemma, for  families of meromorphic functions which are not normal in a domain. Roughly speaking, it says that {\it a non-normal family can be rescaled at small scale to obtain a non-constant meromorphic function in the limit.} This result of Zalcman gave birth to many new normality criteria.
\begin{Zlc} A family $\mathcal F$ of functions meromorphic (analytic) on the unit disc $\Delta$ is not normal if and only if there exist
\begin{enumerate}
\item[$(a)$]{a number $0<r<1$}
\item[$(b)$]{points $z_j, |z_j|<r$}
\item[$(c)$]{functions $\{f_j\}\subseteq \mathcal F$}
\item[$(d)$]{numbers $\rho_j\rightarrow0^+$}
 \end{enumerate}
such that
\begin{equation}\notag
f_j(z_j+\rho_j\zeta)\rightarrow g(\zeta)
\end{equation}
spherically uniformly (uniformly) on compact subsets of $\C$, where $g$ is a non-constant meromorphic (entire) function on $\C$. \end{Zlc}

Recently, Datt and Kumar proved a normality criterion for a family of holomorphic functions in a domain concerning shared holomorphic functions.
\begin{theorem}[\cite{Datt Kumar 15}]\label{Ind J PAM 1}
  Let $\fr$ be a family of holomorphic functions on a  domain  ${D}$ such that all zeros of $f\in \fr$ are of multiplicity at least $k$, where $k$ is a positive integer. Let $a, b$ be holomorphic functions in ${D}$. If  for each $f\in \fr$,
\begin{enumerate}
\item{ $  b(z)\neq 0,  \text{ for all\ } z\in  D,$}
\item{    $ a(z)\neq b(z)$, \text{ and\ } {$b(z)-L_k(a(z))\neq 0,$} }
\item{ {$f(z)=a(z) \text{ if and only if \ } L_k(f(z))=a(z),$}}
\item{ {$ f(z)=b(z) \text { whenever \ } L_k( f(z))=b(z)$, }}
\end{enumerate}where  $L_k(f(z)):=a_k(z)f^{(k)}(z)+a_{k-1}(z)f^{(k-1)}(z)+\ldots+a_1(z)f'(z)+a_0(z)f(z)$.
Then $\fr$ is normal in $D.$
\end{theorem} 
We prove a quasi-normality criterion related to the theorem of Datt and Kumar (cf. Theorem \ref{qthm}).
 \section{Basic Notions and Main Results}
The notion of quasi-normal families was introduced by Montel in 1922 \cite{Montel 22}.
The concept of quasi-normality is not studied as much as the concept of normal families. Let us recall the definition  of quasi-normal family.
\begin{definition}[\cite{Schiff}]
A  family  $\fr$ of meromorphic (analytic) functions on a domain $D$ is called quasi-normal on $D$ if every sequence $\{f_n\}$ of functions in $\fr$ has a subsequence which converges  uniformly  on compact subsets of $D\setminus E$ with respect to the spherical (Euclidean) metric, where the set  $E$ has no accumulation points in $D$. If the set $E=\emptyset$ we say $\fr$ is normal on $D$.
\end{definition}
If the set $E$ can always be taken to satisfy $|E|\leq q$ then we say $\fr$ is  quasi-normal of order $q$ in $D$. $\fr$ fails to be quasi-normal of order $q$ in $D$ when there exists a set $E$ such that $|E|\geq q+1.$  Thus a family $\fr$ is normal in $D$ if and only if it is quasi-normal of order $0$ in $D$.
\begin{example}Let $\fr:=\{nz: z\in\C, n\in \N\}$. This family is not normal in any domain of $\C$ which contains $0$. Also, for each domain in $\C\setminus\{0\}$ this family is normal.Therefore, $\fr$ is a quasi-normal family of order $1$ in any domain which contains $0$.
\end{example}

This notion was further extended by Chuang in an inductive fashion as $Q_m-$normality \cite{Chuang 93}. He used the concept of $m-$th order derived set to define $Q_m-$normality.
\begin{definition}[\cite{Chuang 93}]
Let $D$ be a plane domain and $E\subset D$. Then the set of all accumulation points of $\fr$ in $D$ is called the derived set of order $1$ in $D$ and denoted by $E_D^{(1)}.$ For every $m\geq 2$, the $m-$th order derived set $E_D^{(m)}$ is defined by  $E_D^{(m)}=\left(E_D^{(m-1)}\right)_D^{(1)}$.
\end{definition}
\begin{definition}[\cite{Chuang 93}]
A family $\fr$ of meromorphic (holomorphic) functions on a domain $D$ is $Q_m-$normal ($m=0, 1, 2, \ldots$) if each sequence $\{f_n\}$ of functions in $\fr$ has a subsequence which converges  uniformly  on compact subsets of $D\setminus E$ with respect to the spherical (Euclidean) metric, where the set  $E$ satisfies $E_D^{(m)}=\emptyset.$
 \end{definition}
It follows from the definition that a $Q_0-$normal family is a normal family and a $Q_1-$normal family is a quasi-normal family. If the set $E$ can always be taken to satisfy $|E_D^{(m)}|\leq q$ then we say $\fr$ is  $Q_m-$normal of order $q$ in $D$.     This notion of $Q_m-$normality was further studied by Nevo \cite{Nevo 01, Nevo 05}.

The theory of normal families is much studied by  shared values and functions. This is not the case with quasi-normlaity. To the best of our knowledge there is only one result on quasi-normality and shared values due to Nevo \cite{Nevo 05}, which is as follows.
\begin{theorem}[\cite{Nevo 05}]\label{Nevo 05}
  Let $\fr$ be a family of functions meromorphic on  a domain $D$. Let  $q\geq 1$ be an integer and $a$, $b$ be two distinct complex numbers. Assume that for any $f\in \fr$ there exist points $z_1^{(f)}, z_2^{(f)}, \ldots, z_{n_f}^{(f)}, n_f\leq q,$ belonging to $D$ such that $f$ and $f'$ share $a$ and $b$ in $D\setminus\{z_1^{(f)}, z_2^{(f)}, \ldots, z_{n_f}^{(f)}\}$. Then $\fr$ is a quasi-normal family of order at most $q$ in $D$.
\end{theorem}
It is natural to ask whether we can give a quasi-normality criterion concerning shared functions. Here we propose a quasi-normality criterion for a family of holomorphic functions which extends Theorem \ref{Nevo 05} in some extent. We replace  values $a,  b$  by holomorphic functions $a(z),  b(z)$  respectively and $f'$ by a linear differential polynomial  $a_k(z)f^{(k)}(z)+a_{k-1}(z)f^{(k-1)}(z)+\ldots+a_1(z)f'(z)+a_0(z) f(z)$, where  $a_0(z),\ldots  a_k(z) $ are  holomorphic functions,  with $a_k(z)\neq 0 \ \text{in\ } {D}$. We define $$L_k(f(z)):=a_k(z)f^{(k)}(z)+a_{k-1}(z)f^{(k-1)}(z)+\ldots+a_1(z)f'(z)+a_0(z)f(z).$$  We propose the following result:
\begin{theorem}\label{qthm}
  Let $\mathcal F$ be a family of holomorphic functions on a  domain  ${D}$ such that all zeros of $f\in \fr$ are of multiplicity at least $k$, where $k$ is a positive integer. Let $a,\ b $ be holomorphic functions in ${D}$ such that \begin{enumerate}
            \item{ $  b(z)\neq 0,  \text{ for all\ } z\in  D,$}
\item{    $ a(z)\neq b(z)$ \text{ and\ } {$b(z)-L_k(a(z))\neq 0.$} }
          \end{enumerate}
           If  for each $f\in \fr$ there exist points $z_1^{(f)}, z_2^{(f)}, \ldots, z_{N_f}^{(f)}, N_f\leq q,$ in $D$ such that for all $z \in D\setminus \{z_1^{(f)}, z_2^{(f)}, \ldots, z_{N_f}^{(f)}\}$
\begin{enumerate}
\item[(i)]{ {$f(z)=a(z) \text{ if and only if \ } L_k(f(z))=a(z),$}}
\item[(ii)]{ {$ f(z)=b(z) \text { whenever \ } L_k( f(z))=b(z)$. }}
\end{enumerate}
Then $\fr$ is a quasi-normal family of order at most $q$ in $D.$
\end{theorem}
\begin{proof}
  Let $\{f_n\}$ be  a sequence in $\fr$. Without loss of generality we assume that $N_{f_n}=N$, $0\leq N\leq q.$  If $N=0$, then $\fr$ is normal and hence quasi-normal of order $0$, by Theorem \ref{Ind J PAM 1}. For $N\geq 1$, we have for each $n$, $f_n$ and $L_k(f_n)$ share $a(z)$ and  $ f_n(z)=b(z) \text { whenever \ } L_k( f_n(z))=b(z)$ in the domain $D\setminus \{\zeta_1^{(n)}, \zeta_2^{(n)}, \ldots, \zeta_{N}^{(n)}\}.$ We also  assume that as $k\rightarrow\infty,$ $\zeta_j^{(k)}\rightarrow\zeta_j\in \overline{D},$ for  $1\leq j\leq N$. \\

  Now, we claim that $\{f_n\}$ converges compactly on $D^*=D\setminus \{\zeta_1, \zeta_2, \ldots, \zeta_{N}\}.$ Take a point $\zeta_0\in D\setminus\{\zeta_1, \zeta_2, \ldots, \zeta_{N}\}$ and $r>0$ such that $\overline{D}(\zeta_0, r)=\{z: |z-\zeta_0|\leq r\}\subset D^*.$ For large values of $n$ we have $\zeta_j^{(n)}\not\in D(\zeta_0, r)$ for $1\leq j\leq N.$ Hence $\{f_n\}$ converges compactly on $D(\zeta_0, r)$, by Theorem \ref{Ind J PAM 1}. Thus we deduce that $\fr$ is normal in $D^*$. Since the set of irregular points is of cardinality  at most $q$ therefore by the definition, $\fr$ is a quasi-normal family of order at most $q$ in $D$.\end{proof}
The following example illustrates Theorem \ref{qthm}.
\begin{example}\label{3.3}
  Let $k=1$ and $D=\{z:|z|<1\}$. Let $a(z)=z^{2}$, $b(z)=z+2$ be holomorphic functions such that $a(z)\neq b(z)$, $b(z)\neq 0$ in $D$. Define $D_1(f)= f'(z)$, then $b(z)\neq D_1(a)$ in $D$. Consider the family $\fr$ of holomorphic function in $D$ given by $\fr=\{(n+1)^2z: z\in D, n= 1, 2, \ldots\}.$ For each function $f_n\in\fr$, there is a point $\frac{2}{(n+1)^2-1}$ so that $f_n$ satisfies condition (i) and (ii) of  the theorem in $D\setminus\left\{\frac{2}{(n+1)^2-1}\right\}$. And $\fr$ is  a quasi-normal family of order 1.
\end{example}If we replace $D$ by $D\setminus\{0\}$ in  Example \ref{3.3}, then $\fr$ satisfies all the conditions and $\fr$ is a normal family in $D\setminus\{0\}$, which means $\fr$ is a quasi-normal family of order 0 in $D\setminus\{0\}$.

In 1999, Fang and Hong extended Montel's normality test using the concept of shared set \cite{Fang Hong 99}. They proposed the following Theorem.
\begin{theorem}[\cite{Fang Hong 99}]\label{Fang Hong 99}
 Let $\fr$ be a family of functions meromorphic on  a domain $D$. If each pair $f$ and $g$ in $\fr$  shares the set $S=\{0, 1, \infty\}$ then $\fr$ is normal on $D$.
\end{theorem}
It is a natural question to ask whether one can replace values by holomorphic functions on $S$. The following example confirms that the normality will no longer be assured if we replace values by holomorphic functions on $S$.
\begin{example}
  Let $D=\{z: |z|<1\}$ and $\fr=\{(n+3)z: n\in\N\}$, clearly each pair $f$ and $g$ of $\fr$ shares $a_1(z)=z$, $a_2(z)=2z$ and $a_3(z)=3z$ but $\fr$ is not normal in $D$.
\end{example}
Here we extend Theorem \ref{Fang Hong 99} using  techniques of shared functions.
\begin{theorem}\label{Main2 AnnPolo}
  Let $\fr$ be a family of functions meromorphic in a domain $D\subseteq\C$. Let $a_1, a_2$ and $a_3$ be three distinct holomorphic functions. If each pair $f$ and $g$  of $\fr$ shares $a_1, a_2$ and $a_3$, then $\fr$ is quasi-normal in $D$.
\end{theorem}
The following example elucidates Theorem \ref{Main2 AnnPolo}.
\begin{example}
  Let $D=\{z: |z|<1\}$ and $\fr=\{(2n+1)z: n\in\N\}$, clearly each pair $f$ and $g$ of $\fr$ shares $a_1(z)=z/2$, $a_2(z)=z$ and $a_3(z)=2z$ and $\fr$ is quasi-normal in $D$.
\end{example}
We use the following results in order to prove Theorem \ref{Main2 AnnPolo}.
\begin{lemma}[\cite{Chang 09}]\label{ChangT 09}
Let $\fr$ be a family of meromorphic functions on a plane domain $D$ and let $\alpha_1$, $\alpha_2$ and $\alpha_3$ be distinct meromorphic functions on $D$, one of which may be $\infty$ identically. If for each $f\in\fr$ and $z\in D$, $f(z)\neq \alpha_i(z),$ for all $i=1, 2, 3,$ then  $\fr$ is normal on $D$.
\end{lemma}
\begin{lemma}[\cite{Chang 09}]\label{ChangL 09}
Let $\fr$ be a family of meromorphic functions on  $\Delta$ and let $\alpha_1$, $\alpha_2$  be distinct holomorphic functions on $\Delta$. Suppose that  for each $f\in\fr$ and $z\in D$, $f(z)\neq \alpha_i(z),$ for all $i=1, 2$. If $\fr$ is normal in $\Delta'=\{z: 0<|z|<1\}$, then $\fr$ is normal in $D$.

\end{lemma}
\begin{proof}[Proof of Theorem \ref{Main2 AnnPolo}]
  Since normality is a local property without loss of generality we may choose $D:=\Delta$. By the assumption, for each $j\in\{1, 2, 3\},$  the set $X=\cup_{j=1}^{3}\{z\in D: f(z)-a_j(z)=0\}$ does not depend on the mapping $f\in\fr.$

Clearly, $X$ is an isolated set in $D$ otherwise $f\equiv a_j$ for one of the $j\in\{1, 2, 3\}$. Also, for any fixed point $z_0\in D\setminus X$, there exists an open neighborhood $N_{\epsilon}(z_0)$ in $D\setminus X$ such that for all $n\geq1$
  \begin{equation}\label{eq2 AnnPolo}
    X \cap N_{\epsilon}(z_0)=\emptyset.
  \end{equation}
  We now prove that $\left\{f_n|_{N_{\epsilon}(z_0)}\right\}$ is a normal family on $N_{\epsilon}(z_0)$. \\

  Assume that $\left\{f_n|_{N_{\epsilon}(z_0)}\right\}$ is not normal in $N_{\epsilon}(z_0)$, then by Zalcman's lemma, there exist a subsequence  $\left\{f_n|_{N_{\epsilon}(z_0)}\right\}$ (after renumbering), a sequence of points $\{z_n\}\subset N_{\epsilon}(z_0)$ such that $\{z_n\}\rightarrow z_0$ and a sequence of positive real numbers $\{\rho_n\}\rightarrow0$ such that the sequence\begin{equation*}
           g_n(\xi)=f_n(z_n+\rho_n\xi)                                                                                                 \end{equation*}
           converges uniformly on compact subsets of $\C$ to a non-constant meromorphic function $g:\C\rightarrow\C$.

         Now, since $f_n(z_n+\rho_n\xi)-a_j(z_n+\rho_n\xi)\neq 0$, thus by Hurwitz's theorem, for each $j\in\{1, 2, 3\}$ we have either $g(\xi)\neq a_j(z_0), $ for all $\xi\in\C$ or $g_j(\xi)\equiv a_j(z_0),$ for all $\xi\in\C$. If $g_j(\xi)\equiv a_j(z_0)$ for one of the $j\in\{1, 2, 3\}$. Then $g$ is constant. Otherwise $g_j(\xi)\neq a_j(z_0)$ for all $j\in\{1, 2, 3\}.$ If all $a_j(z_0), j=1, 2, 3$ are distinct by Picard's theorem $g$ is constant. Otherwise we have two cases to consider

         When two of $a_j(z_0)$, $j=1, 2, 3$ are equal. Without any loss of generality we may assume that $a_1(z_0)=a_2(z_0)$ and $a_3(z_0)$ is distinct from $a_j(z_0)$ for $j=1, 2$. Since for all $f\in\fr$,  $z\in B(z_0, \epsilon)$ and $j\in\{1, 2, 3\}$,    $f(z)\neq a_j(z)$. Also $a_j$ are holomorphic in $D$, we get that $a_i(z)\neq a_j(z)$, $1\leq i, j\leq 3$ in the deleted neighborhood $N'_{\epsilon}(z_0)$ of $z_0$. Hence $\fr$ is normal in $N'_{\epsilon}(z_0)$. Then by Lemma \ref{ChangL 09} $\fr$ is normal in $N_{\epsilon}(z_0).$ The last case is when $a_1(z_0)=a_2(z_0)=a_3(z_0)$. Using Lemma \ref{ChangT 09}, in this case $\fr$ is normal in $N_{\epsilon}(z_0)$.

           Thus by the usual diagonal argument we can find a subsequence (again denoted by $\{f_n\}$) which converges uniformly on compact subsets of $D\setminus X$ to a meromorphic function $f$ of $D\setminus X$. Hence $\{f_n\}$ is quasi-normal in $D$.
\end{proof}

\end{document}